\documentclass[]{article}
\begin{document}
\newtheorem{proposition}{Proposition}[section]
\newtheorem{definition}{Definition}[section]
\newtheorem{lemma}{Lemma}[section]

\title{\bf Basic Zero Divisor Mathematics (I)}
\author { Keqin Liu\\Department of Mathematics\\The University of British Columbia\\Vancouver, BC\\
Canada, V6T 1Z2}
\date{July, 2026}
\maketitle

\begin{abstract} We  develop refinements of  inner products, the Gram-Schmidt process,  root system, Hilbert space and 
Maschke's Theorem in the context of zero divisor mathematics.  We then introduce  
the concept of  respecting nilpotent step representations of Lie algebras and give the new refinement
 of Ado's Theorem.
\end{abstract}

\bigskip\bigskip
Zero divisor mathematics study the objects which are obtained from some well-known objects in mathematics  by replacing the vector spaces over fields with the modules over some commutative associative algebras with zero divisors. Among the commutative associative algebras which can be used to develop the theory of zero divisor mathematics, the simplest one is the dual real number algebra $R^{(2)}$  introduced by W. K. Clliford in 1873. 

\medskip
Replacing the ordinary differentiability based on the real linear transformations with the  $R^{(2)}$-differentiability
based on the $R^{(2)}$-module maps, we have developed  zero divisor calculus  which extend the basic theory of calculus on manifolds in \cite{S}.
In addition to the new objects in  zero divisor calculus (\cite{KL}),  there are many other new objects worthy to be noted in zero divisor mathematics. 
In this paper, we explain how to refine some well-known objects and theorems from a few different mathematics areas in the context of zero divisor mathematics.

\medskip
This paper consists of the following four sections and an appendix:

\bigskip 
Section 1\,\, Basic zero divisor linear algebra,

Section 2\,\, $R^{(2)}$-root system ,

Section 3\,\, Real Hilbert-like spaces,

Section 4\,\, Zero divisor representations of finite groups,

Appendix\,\, A class of  representations of non-semisimple Lie algebras.

\medskip
In section 1, after discussing  the  basic properties of the modules over the dual real number algebra
$R^{(2)}$, we extend the important concept of inner product, generalize the 
Gram-Schmidt process  and the orthonormal basis for vector spaces in the context of $R^{(2)}$-modules. 
In section 2, we introduce the concept of $R^{(2)}$-root system  based on two types of reflections and prove the existence of simple $R^{(2)}$-root system. In section 3, 
we initiate the study of real Hilbert-like spaces.  Our motivation of studying  Hilbert-like spaces comes from our belief
that there exists a class of Hilbert-like spaces such that the class of Hilbert-like spaces can replace Hilbert spaces
in the numerous applications of Hilbert spaces and the famous invariant subspace problem becomes a trivial problem  in the class of Hilbert-like spaces.  The main result in section 3 is the counterpart of the Riesz representation theorem in the context of  real Hilbert-like spaces.

\medskip
In section 4 and the appendix, we present two new ways of doing representation theory. The key ideal in our approach is to replace the linear transformations in current representation theory with the linear transformations which have non-trial invariant subspaces. There are two basic ways of finding  the linear transformations which have non-trial invariant subspaces to reconsider the current representation theory.
The first  way is to use the module maps over some commutative associative algebras with zero divisor, and the second way is to use the extension of the triangular linear transformations on
the graded vector space in \cite{KL1}.  We use the first way to introduce the concept of zero divisor representations of finite groups in section 4, and use the second way to introduce 
a class of  representations of non-semisimple Lie algebras and  present the new refinement  of Ado's theorem in the appendix. 

\bigskip
\section{Basic zero divisor linear algebra}

The {\bf dual real number algebra} $\mathcal{R}^{(2)}=\mathcal{R}1\oplus \mathcal{R}1^{\#}$  is the 
2-dimensional  associative real algebra, where  $\{1, \,1^{\#}\}$ is a basis 
for the 2-dimensional real vector space  $\mathcal{R}^{(2)}$ and  the associative product on 
$\mathcal{R}^{(2)}$ is defined by
$$
(x_1+x_2\,1^{\#})(y_1+y_2\,1^{\#})=x_1y_1+(x_1y_2+x_2y_1)\,1^{\#},
$$
where $x_1$, $x_2$, $y_1$, $y_2\in \mathcal{R}$. If 
$x=x_1+x_2\,1^{\#}\in \mathcal{R}^{(2)}$ with $x_1$, $x_2\in \mathcal{R}$,
then $Re\,x:=x_1$ and $Ze\,x:=x_2$ are called the {\bf  real part} and the  {\bf zero-divisor part} of $x$, respectively. The dual 
real number algebra $\mathcal{R}^{(2)}$ has many zero-divisors.
In fact, if $0\ne x\in \mathcal{R}^{(2)}$, then $x$ is a zero-divisor if and only if $Re\,x=0$,  and
$x$ is invertible if and only if $Re\,x\ne 0$. Moreover if $x$ is invertible,  then  the inverse $\displaystyle\frac{1}{x}:=x^{-1}$ of $x$ is given by
$\displaystyle\frac{1}{x}=x^{-1}=\displaystyle\frac{1}{Re\, x}-\displaystyle\frac{Ze\, x}{(Re\, x)^2}\,1^{\#}$. For convenience,  the product $x^{-1}y$ is also denoted by $\displaystyle\frac{y}{x}$, where $x$, $y\in \mathcal{R}^{(2)}$ and  $x$ is invertible.
The  dual real number algebra $\mathcal{R}^{(2)}$ is a normed algebra with respect to the following norm:
$$
||x||:=\sqrt{2\,(Re\,x)^2+(Ze\,x)^2}\quad\mbox{for $x\in\mathcal{R}^{(2)}$.}
$$

\medskip
A  $\mathcal{R}^{(2)}$-module $V$ has the following submodule chain:
$$\{0\}\subseteq Im\,1^{\#}_V \subseteq Ker\,1^{\#}_V\subseteq V,$$
where the submodules $Im\,1^{\#}_V $ and $ Ker\,1^{\#}_V$ are defined by
$$
Im\,1^{\#}_V:=\{ 1^{\#}\,v\,|\, v\in V\}, \qquad 
Ker\,1^{\#}_V:=\{ v\,|\, \mbox{$1^{\#}\,v=0$ and $v\in V$}\}.
$$
An element of a $\mathcal{R}^{(2)}$-module $V$ is  called a {\bf dual real vector} of $V$.  We say that a
  $\mathcal{R}^{(2)}$-module $V$ is a {\bf trivial} if $V=Ker\,1^{\#}_V$. Obviously, a trivial  $\mathcal{R}^{(2)}$-module is 
nothing more than a Euclidean space. 

\medskip
We now give the following generalization of the concept of bases in Euclidean space.

\medskip
\begin{definition}\label{def1} A subset $\beta$ of  a 
$\mathcal{R}^{(2)}$-module $V$ is called a 
{\bf $(\mathcal{R}^{(2)}, \mathcal{R})$-basis} for $V$ if $\beta$ has the following three properties:
\begin{description}
\item[(i)] $\beta=(S_1\,\|\, S_2):=S_1\cup S_2$, where $S_1\subseteq V\setminus  Ker\,1^{\#}_V$  and 
$S_2\subseteq   Ker\,1^{\#}_V\setminus Im\,1^{\#}_V$;
\item[(ii)]  $\beta=(S_1\,\|\, S_2)$ is  {\bf $(\mathcal{R}^{(2)}\,\|\, \mathcal{R})$-linearly independent}, i.e., for any 
finite subset $(\{v_i^{(2)}\}_{i=1}^s\,\|\, \{v_{s+j}\}_{j=1}^t)$ of $\beta$ with $\{v_i^{(2)}\}_{i=1}^s\subseteq S_1$ and 
$\{v_{s+j}\}_{j=1}^t\subseteq S_2$, we have
\begin{eqnarray*}
&&\sum_{i=1}^{s}\,c_i^{(2)}v_i^{(2)}+\sum_{j=1}^{t}\,c_{s+j}v_{s+j}=0,\,  \{c_i^{(2)}\}_{i=1}^{s}\subseteq \mathcal{R}^{(2)}, 
\{c_{s+j}\}_{j=1}^{t}\subseteq \mathcal{R}\\
&&\Longrightarrow c_i^{(2)}=0 \quad\mbox{for $1\le i\le s$}\quad\mbox{and}\quad  c_{s+j}=0 \quad\mbox{for $1\le j\le t$}\;
\end{eqnarray*}
\item[(iii)] For any dual real vector $v$ of $V$, $v$ is a 
{\bf $(\mathcal{R}^{(2)}\,\|\, \mathcal{R})$-linearly combination} of 
$\beta=(S_1\,\|\, S_2)$, i.e., 
there exist  non-negative integers $s$, $t$ and 
$\{c_i^{(2)}\}_{i=1}^{s}\subseteq \mathcal{R}^{(2)}$,  
$\{c_{s+j}\}_{j=1}^{t}\subseteq \mathcal{R}$, $ \{v_i^{(2)}\}_{i=1}^s\subseteq S_1$,
$\{v_{s+j}\}_{j=1}^t\subseteq S_2$ such that
 $ v=\sum_{i=1}^{s}\,c_i^{(2)}v_i^{(2)}+\displaystyle\sum_{j=1}^{t}\,c_{s+j}v_{s+j}$.
\end{description}
\end{definition}

\medskip
The first significant result for $\mathcal{R}^{(2)}$-modules is the following

\medskip
\begin{proposition}\label{pr1} Every $\mathcal{R}^{(2)}$-module has a 
$(\mathcal{R}^{(2)}, \mathcal{R})$-basis.
\end{proposition}

\medskip
\noindent
{\bf Proof} This proposition follows from the existence of a basis for a vector space.

\hfill\raisebox{1mm}{\framebox[2mm]{}}

\bigskip 
{\bf Remark} A corollary of Proposition \ref{pr1} is that if $V$ is  $\mathcal{R}^{(2)}$-module, $u\in V$, $ 0\ne r^{(2)}\in \mathcal{R}^{(2)}$
and $ r^{(2)}u=0$, then
$u\left\{\begin{array}{ll} =0& \mbox{if $Re\,r^{(2)}\ne 0$,} \\\in Ker\,1^{\#}_V &\mbox{if $Re\,r^{(2)}= 0$.} \end{array}\right.$

\bigskip 
If a $\mathcal{R}^{(2)}$-module $V$ has a finite $(\mathcal{R}^{(2)}, \mathcal{R})$-basis
$(\{v_i\}_{i=1}^n\,\|\,\{v_{n+j}\}_{j=1}^m)$, then $V$ is called {\bf finite dimensional}, the 
unique pair $(n, m)$ of non-negative integers is called the $\mathcal{R}^{(2)}$-{\bf dimension} of $V$ 
and is denoted by 
$\dim_{\mathcal{R}^{(2)}}(V)=(n,\,m)$. 

\medskip 
Let  $n$ and $m$ be non-negative integers with $n+m\ne 0$. Let
$$\mathcal{R}^{(2)n,m}:=
\left\{(x^1,  \dots ,  x^n\,| \,x^{n+1}1^{\#},  \dots ,  x^{n+m}1^{\#})\left |
\begin{array}{c}
x^1, \dots, x^n\in \mathcal{R}^{(2)}\\ 
x^{n+1},  \dots ,  x^{n+m}\in \mathcal{R}\end{array}\right.\right\}.$$
Then $\mathcal{R}^{(2)n,m}$  is a  $\mathcal{R}^{(2)}$-module with the operation of coordinate-wise addition and scalar multiplication. An element in $\mathcal{R}^{(2)n,m}$ is called a {\bf real $(n, m)$-vector}, and  
$\mathcal{R}^{(2)n,m}$ is called 
the {\bf real $(n, m)$-dimensional $\mathcal{R}^{(2)}$-module}. Let
\begin{eqnarray*}
e_1^{(2)}:&=&(\underbrace{1, 0, \dots ,  0, 0}_{n}\,| \,\underbrace{0, 0, \dots , 0, 0}_{m})\\
&\vdots&\\
e_n^{(2)}:&=&(\underbrace{0, 0, \dots ,  0, 1}_{n}\,| \,\underbrace{0, 0, \dots , 0, 0}_{m})\\
e_{n+1}:&=&( \underbrace{0, 0\dots ,  0, 0}_{n}\,| \,\underbrace{1^{\#}, 0, \dots , 0, 0}_{m})\\
&\vdots&\\
e_{n+m}:&=&( \underbrace{0, 0\dots ,  0, 0}_{n}\,| \,\underbrace{0, 0, \dots , 0, 1^{\#}}_{m}).
\end{eqnarray*}
Then $(\{e_i^{(2)}\}_{i=1}^n\,\|\, \{e_{n+j}\}_{j=1}^m)$ is a 
$(\mathcal{R}^{(2)}, \mathcal{R})$-basis, which is called  the {\bf standard $(\mathcal{R}^{(2)}, \mathcal{R})$-basis} for the  real $(n, m)$-module $\mathcal{R}^{(2)n, m}$.  It is easy to check that a $\mathcal{R}^{(2)}$-module $V$ is isomorphic to
 $\mathcal{R}^{(2)n, m}$ if and only if $\dim_{\mathcal{R}^{(2)}}(V)=(n,\,m)$.

\bigskip
Let $a\in \mathcal{R}^{(2)}$. A dual real number $x\in\mathcal{R}^{(2)}$ is called a {\bf square root} of $a$ if $x^2=a$. The square root of $a$ with positive real part is called {\bf positive square root} of $a$ and is denoted by $\sqrt{a}$. Clearly, 
the  positive square root $\sqrt{a}$ of $a$ exists if and only if $Re\,a>0$, and 
$\sqrt{a}=\sqrt{Re\, a}+\displaystyle\frac{Ze\, a}{2\sqrt{Re\, a}}\,1^{\#}$.

\bigskip
One of the important concepts in zero divisor linear algebra is  the following generalization of the ordinary inner product on a real vector space.

\medskip
\begin{definition}\label{def1.2} Let $V$ be a left $\mathcal{R}^{(2)}$-module. 
A {\bf real Hu-Liu inner product} on $V$ is a function
$\kappa: V\times V\to \mathcal{R}^{(2)}$ such that for all $x$, $y$, and $z\in V$ and each dual real number $a\in \mathcal{R}^{(2)}$, the following five properties hold:
\begin{description} 
\item[(i)] $\kappa (x+y, z)=\kappa (x, z)+\kappa (y, z)$.
\item[(ii)] $\kappa (ax, y)=a\,\kappa (x, y)$.
\item[(iii)]  $\kappa (x, y)=\kappa (y, x)$.
\item[(iv)] $Re\,\kappa (x, x)\,\ge 0$, and $Re\,\kappa (x, x)\,= 0$ if and only if 
$x\in Ker\,1^{\#}_V$.
\item[(v)]  $Ze\,\kappa (x, x)\,\ge 0$ for  $x\in Ker\,1^{\#}_V$, and $Ze\,\kappa (x, x)= 0$ if and only if 
$x\in Im\,1^{\#}_V$.
\end{description}
\end{definition}

\bigskip
A  $\mathcal{R}^{(2)}$-module $V$ endowed with a specific real Hu-Liu inner product $\kappa (\,, \,)$ is called a {\bf real Hu-Liu inner product  $\mathcal{R}^{(2)}$-module} and is denoted by $(V,  \kappa (\,, \,))$.

\medskip
\begin{definition}\label{def1.3} Let $(V,  \kappa (\,, \,))$ be a real Hu-Liu inner product   
$\mathcal{R}^{(2)}$-module. 
\begin{description}
\item[(i)] Two dual real vectors $x$ and $y$ in $V$ are  
{\bf $\mathcal{R}^{(2)}$-orthogonal} if $\kappa (x, y)=0$.  
\item[(ii)] A subset $S$ of $V$ is 
{\bf } {\bf $\mathcal{R}^{(2)}$-orthogonal} if any two distinct dual real vectors in $S$ are 
$\mathcal{R}^{(2)}$-orthogonal. 
\item[(iii)] A dual real vectors $x$ in $V$  is a {\bf $\mathcal{R}^{(2)}$-unit 
vector} if 
$$\kappa (x, x)=\left\{\begin{array}{ll}1&\mbox{for $x\in V\setminus Ker\,1^{\#}_V$,}\\
1^{\#}&\mbox{for $x\in Ker\,1^{\#}_V\setminus Im\,1^{\#}_V$}.\end{array}\right.$$
\item[(iv)] A subset $S$ of $V$ is {\bf $\mathcal{R}^{(2)}$-orthonormal} if $S$ is orthogonal and consists  entirely of $\mathcal{R}^{(2)}$-unit dual real 
vectors.
\end{description}
\end{definition}

\bigskip
It is easy the check that the function 
$\kappa_{n,m}: \mathcal{R}^{(2)n,m}\times \mathcal{R}^{(2)n,m}\to \mathcal{R}^{(2)}$ defined by
\begin{eqnarray}\label{eq1}
&&\kappa_{n,m} (x, y):=\displaystyle\sum_{i=1}^n x^iy^i+
1^{\#}\displaystyle\sum_{j=1}^m  x^{n+j}y^{n+j}\nonumber\\
&=&\displaystyle\sum_{i=1}^n x^{i1}y^{i1}+
1^{\#}\Big(\displaystyle\sum_{i=1}^n(x^{i1}y^{i2}+x^{i2}y^{i1})+\displaystyle\sum_{j=1}^m  x^{n+j}y^{n+j}\Big)
\end{eqnarray}
is a Hu-Liu inner product on $\mathcal{R}^{(2)n,m}$, where $x$ and $y$ are given by 
\begin{equation}\label{eq2}
x=\displaystyle\sum_{i=1}^n (\underbrace{x^{i1}+1^{\#}x^{i2}}_{x^i})e_i^{(2)}+
\sum_{j=1}^m\,x^{n+j}e_{n+j}\in \mathcal{R}^{(2)n,m},
\end{equation}
\begin{equation}\label{eq3}
y=\displaystyle\sum_{i=1}^n (\underbrace{y^{i1}+1^{\#}y^{i2}}_{y^i})e_i^{(2)}+
\sum_{j=1}^m\,y^{n+j}e_{n+j}\in \mathcal{R}^{(2)n,m},
\end{equation}
where  $x^{i1}$, $x^{i2}$, $y^{i1}$, $y^{i2}$, $x^{n+j}$,  $y^{n+j}\in \mathcal{R}$ with
$1\le i\le n$ and $1\le j\le m$,  and $(\{e_i^{(2)}\}_{i=1}^n\,\|\, \{e_{n+j}\}_{j=1}^m)$ is  the standard 
$(\mathcal{R}^{(2)}, \mathcal{R})$-basis for the  real $(n, m)$-module $\mathcal{R}^{(2)n, m}$.
The Hu-Liu inner product $\kappa_{n,m} (\, , \,)$ defined by (\ref{eq1}) is called the {\bf standard Hu-Liu inner product} on the
 real $(n, m)$-dimensional $\mathcal{R}^{(2)}$-module  $\mathcal{R}^{(2)n,m}$.

\medskip
Our extension of the well-known Gram-Schmidt process is given in the following
 
\medskip
\begin{proposition}\label{pr1.3} Let $(V,  \kappa (\,, \,))$ be a real Hu-Liu inner product   
$\mathcal{R}^{(2)}$-module. If $\beta=(\{w_i^{(2)}\}_{i=1}^n\,\|\,\{w_{n+j}\}_{j=1}^m)$ is a 
$(\mathcal{R}^{(2)}\,\|\, \mathcal{R})$-linearly independent subset of a $\mathcal{R}^{(2)}$-module $V$, there exists a 
 $\mathcal{R}^{(2)}$-orthonormal set $\gamma=(\{v_i^{(2)}\}_{i=1}^n\,\|\,\{v_{n+j}\}_{j=1}^m)$ of $V$ such that 
 $span^\#(\gamma)=span^\#(\beta)$, where $span^\#(\beta)$ is the 
{\bf $\mathcal{R}^{(2)}$-span} of the set $\beta$ and is defined by
$$
span^\#(\beta)=\sum_{i=1}^n  \mathcal{R}^{(2)}w_i^{(2)} +\sum_{j=1}^m \mathcal{R}^{(2)}w_{n+j}
=\sum_{i=1}^n  \mathcal{R}^{(2)}w_i^{(2)} +\sum_{j=1}^m \mathcal{R}w_{n+j}
$$
for $\beta=(\{w_i^{(2)}\}_{i=1}^n\,\|\,\{w_{n+j}\}_{j=1}^m)$.
\end{proposition}

\medskip
\noindent
{\bf Proof}  See \cite{KL} for the poof.

\hfill\raisebox{1mm}{\framebox[2mm]{}}

\medskip
The process of constructing the $\mathcal{R}^{(2)}$-orthogonal set $\gamma$ from the 
$(\mathcal{R}^{(2)}\,\|\, \mathcal{R})$-linearly independent $\beta$ in Proposition \ref{pr1.3} is called 
the {\bf Hu-Liu extension of the Gram-Schmidt process}.

\medskip
\begin{definition}\label{def1.4} A 
$\mathcal{R}^{(2)}$-submodule W of a $\mathcal{R}^{(2)}$-module $V$ is said to be 
{\bf pleasant} if $W\bigcap Im\,1^{\#}_V=Im\,1^{\#}_W$, where 
$Im\,1^{\#}_W:=\{1^{\#}\,x\,|\, x\in W\}.$
\end{definition}

\medskip
Since $V\bigcap Im\,1^{\#}_V=Im\,1^{\#}_V$ and $\{0\}\bigcap Im\,1^{\#}_V=\{0\}=Im\,1^{\#}_{\{0\}}$, both $\{0\}$ and $V$ are 
pleasant submodules of a $\mathcal{R}^{(2)}$-module $V$. If $V$ is a {\bf non-trivial} $\mathcal{R}^{(2)}$-module, i.e., 
$Im\,1^{\#}_V\ne \{0\}$, then both $Ker\,1^{\#}_V$ and $Im\,1^{\#}_V$ are not pleasant submodules because
$$
\clubsuit\bigcap Im\,1^{\#}_V= Im\,1^{\#}_V\ne \{0\}=Im\,1^{\#}_{\clubsuit}
\quad\mbox{for $\clubsuit\in \{Ker\,1^{\#}_V, \,\, Im\,1^{\#}_V$\}.}
$$

\medskip
The next proposition gives a class of  pleasant submodules of a $\mathcal{R}^{(2)}$-module.

\medskip
\begin{proposition}\label{pr1.2}  if $x\ne 0$ is a non-zero element of a $\mathcal{R}^{(2)}$-module $V$, then the $\mathcal{R}^{(2)}$-submodules
$\mathcal{R}^{(2)}x$ of $V$ is pleasant if and only if $x\not\in  Im\,1^{\#}_V$.
\end{proposition}

\medskip
\noindent
{\bf Proof} If $x\in Im\,1^{\#}_V$, then 
$
\mathcal{R}^{(2)}x\bigcap Im\,1^{\#}_V=\mathcal{R}^{(2)}x\ne \{0\}= Im\,1^{\#}_{\mathcal{R}^{(2)}x},
$
which proves that  the $\mathcal{R}^{(2)}$-submodule $\mathcal{R}^{(2)}x$ is not  pleasant. 

\medskip
Conversely, if $x\not\in  Im\,1^{\#}_V$, then
either $x\not\in Ker\,1^{\#}_V$ or $x\in Ker\,1^{\#}_V\setminus  Im\,1^{\#}_V$. In the case where $x\not\in Ker\,1^{\#}_V$, we have
\begin{eqnarray*}
&& 1^{\#}v \in \mathcal{R}^{(2)}x\bigcap Im\,1^{\#}_V\quad\mbox{for some $v\in V$}\\
&\Rightarrow&  1^{\#}v=a^{(2)}x \quad\mbox{for some $a^{(2)}\in \mathcal{R}^{(2)}$}\\
&\Rightarrow&  0= 1^{\#}(1^{\#}v)=1^{\#}(a^{(2)}x )=(Re\,a^{(2)} )(1^{\#}x)\\
&\Rightarrow& Re\,a^{(2)}=0\quad\mbox{because $1^{\#}x\ne 0$ }\\
&\Rightarrow& 1^{\#}v=a^{(2)}x=(Re\, a^{(2)}+1^{\#}Ze\, a^{(2)})x =(Ze\,a^{(2)})( 1^{\#}x)\\
&\Rightarrow&1^{\#}v=1^{\#}((Ze\,a^{(2)})x)\in Im\,1^{\#}_{\mathcal{R}^{(2)}x},
\end{eqnarray*}
which implies that $\mathcal{R}^{(2)}x\bigcap Im\,1^{\#}_V\subseteq  Im\,1^{\#}_{\mathcal{R}^{(2)}x}$. This proves that 
 the $\mathcal{R}^{(2)}$-submodule $\mathcal{R}^{(2)}x$ is  pleasant for $x\not\in  Ker\,1^{\#}_V$.

In the case where  $x\in Ker\,1^{\#}_V\setminus  Im\,1^{\#}_V$,  we have
\begin{eqnarray*}
&&y\in  \mathcal{R}^{(2)}x\bigcap Im\,1^{\#}_V= \mathcal{R}x\bigcap Im\,1^{\#}_V\\
&\Rightarrow& y=1^{\#}v=ax\quad\mbox{for some $v\in V$ and $a\in  \mathcal{R}$}\\
&\Rightarrow& a=0 \quad\mbox{(because $a\ne 0$ gives $x=a^{-1}1^{\#}v\in Im\,1^{\#}_V$, a contradiction)}\\
&\Rightarrow& y=ax=0.
\end{eqnarray*}
It follows  that $\mathcal{R}^{(2)}x\bigcap Im\,1^{\#}_V=\{0\}=Im\,1^{\#}_{\mathcal{R}^{(2)}x}$. This proves that 
 the $\mathcal{R}^{(2)}$-submodule $\mathcal{R}^{(2)}x$ is also pleasant for $x\in Ker\,1^{\#}_V\setminus  Im\,1^{\#}_V$.

\hfill\raisebox{1mm}{\framebox[2mm]{}}

\bigskip
The following proposition gives the fundamental fact concerning  pleasant $\mathcal{R}^{(2)}$-submodules.

\medskip
\begin{proposition}\label{pr1.4} Let $W$ be a  $\mathcal{R}^{(2)}$-submodule a finite dimensional
$\mathcal{R}^{(2)}$-module $V$. Then the following are equivalent:
\begin{description} 
\item[(i)] $W$ is pleasant.
\item[(ii)] There exist 
$w_1^{(2)}, \dots, w_s^{(2)}\in W\setminus  Ker\,1^{\#}_V$ and  
$w_{s+1}, \dots, w_{s+t}\in (W\bigcap Ker\,1^{\#}_V)\setminus  Im\,1^{\#}_V$
such that the set 
$S_W:=(\{w_i^{(2)}\}_{i=1}^s\,\|\, \{w_{s+j}\}_{j=1}^t)$
is a  $(\mathcal{R}^{(2)}, \mathcal{R})$-basis of $W$, and the set $S_W$ can be extended to a 
$(\mathcal{R}^{(2)}, \mathcal{R})$-basis $S_V$ of the  $\mathcal{R}^{(2)}$-module $V$with
$$
S_V:=(\{w_i^{(2)}\}_{i=1}^s\cup \{w_{s+k}^{(2)}\}_{k=1}^n\,\|\, \{w_{s+j}\}_{j=1}^t\cup \{w_{s+t+\ell}\}_{\ell=1}^m).
$$
\end{description}
\end{proposition}

\medskip
\noindent
{\bf Proof} Let us prove (ii) $\Rightarrow$ (i). If (ii) holds,  we have
\begin{equation}\label{eq54}
V=\left(\bigoplus _{i=1}^{s+n} \mathcal{R}^{(2)}w_i^{(2)}\right)\bigoplus
\left(\bigoplus _{j=1}^{t+m} \mathcal{R}w_{s+j}\right)=W\bigoplus U\quad\mbox{as $\mathcal{R}^{(2)}$-modules},
\end{equation}
where 
\begin{equation}\label{eq55}
W=\left(\bigoplus _{i=1}^s \mathcal{R}^{(2)}w_i^{(2)}\right)\bigoplus
\left(\bigoplus _{j=1}^t \mathcal{R}w_{s+j}\right)\quad\mbox{as $\mathcal{R}^{(2)}$-modules}
\end{equation}
and
\begin{equation}\label{eq56}
U=\left(\bigoplus _{i=s+1}^{s+n} \mathcal{R}^{(2)}w_i^{(2)}\right)\bigoplus
\left(\bigoplus _{j=t+1}^{t +m}\mathcal{R}w_{s+j}\right)\quad\mbox{as $\mathcal{R}^{(2)}$-modules.}
\end{equation}

To prove that $W$ is pleasant, it is enough to prove 
$W\bigcap Im\,1^{\#}_V\subseteq Im\,1^{\#}_W$. If $x\in W\bigcap Im\,1^{\#}_V$, then $x=1^{\#} v$ for some 
$v\in V$. By  (\ref{eq54}), there exist $\{A_i^{(2)}\}_{i=1}^{s+n}\subseteq \mathcal{R}^{(2)}$  and 
$ \{A_{s+j}\}_{j=1}^{t+m} \subseteq \mathcal{R}$ such that
$v=\displaystyle\sum_{i=1}^{s+n} A_i^{(2)}w_i^{(2)}+\displaystyle\sum_{j=1}^{t+m} A_{s+j}w_{s+j}$. Hence, we get
$W\ni x=1^{\#} v=\displaystyle\sum_{i=1}^{s+n} (Re\, A_i^{(2)})(1^{\#}w_i^{(2)})$.
It follows from (\ref{eq55}) and (\ref{eq56}) that
$$
W\ni 1^{\#} v-\displaystyle\sum_{i=1}^{s} (Re\, A_i^{(2)})(1^{\#}w_i^{(2)})=
\displaystyle\sum_{i=s+1}^{s+n} (Re\, A_i^{(2)})(1^{\#}w_i^{(2)})\in U,
$$
which implies that $1^{\#} v-\displaystyle\sum_{i=1}^{s} (Re\, A_i^{(2)})(1^{\#}w_i^{(2)})\in W\bigcap U=\{0\}$ or
$$x=1^{\#} v=\displaystyle\sum_{i=1}^{s} (Re\, A_i^{(2)})(1^{\#}w_i^{(2)})
=1^{\#} \displaystyle\sum_{i=1}^{s} (Re\, A_i^{(2)})w_i^{(2)}\in  Im\,1^{\#}_W.$$
This proves $W\bigcap Im\,1^{\#}_V\subseteq Im\,1^{\#}_W$. Hence, (i) holds.

\hfill\raisebox{1mm}{\framebox[2mm]{}}

\medskip
\begin{definition}\label{def1.5}
Let $(V,  \kappa (\,, \,))$ be a  real Hu-Liu inner product  $\mathcal{R}^{(2)}$-module, and let $S$ be a non-empty subset 
of $V$. We define $S_{\perp}$ to be the set of all dual real vectors in $V$ that are $\mathcal{R}^{(2)}$-orthogonal to every dual real vector in $S$; that is, 
$$
S_{\perp}:=\{x\in V\,|\, \kappa (x, y)=0\quad\mbox{for all $y\in S$}\}.
$$
The set $S_{\perp}$ is called the {\bf $\mathcal{R}^{(2)}$-orthogonal complement} of $S$.
\end{definition}

\medskip
Clearly, $S_{\perp}$ is a $\mathcal{R}^{(2)}$-submodule of $V$ for any subset $S$ of $V$.

\medskip
\begin{proposition}\label{pr1.5} If $W$ is a  pleasant $\mathcal{R}^{(2)}$-submodule a finite dimensional
$\mathcal{R}^{(2)}$-module $V$, then $V=W\oplus W_{\perp}$ as $\mathcal{R}^{(2)}$-modules.
\end{proposition}

\medskip
\noindent
{\bf Proof} Use Proposition \ref{pr1.4} and the  Hu-Liu extension of the Gram-Schmidt process.

\hfill\raisebox{1mm}{\framebox[2mm]{}}

\bigskip
Finally, we introduce  Hu-Liu matrices and present the basic facts about Hu-Liu matrices. In the following of this section,  
$\mathcal{K}$  denotes  a field. The {\bf dual  algebra over $\mathcal{K}$ } is denoted by $\mathcal{K}^{(2)}$, where
$\mathcal{K}^{(2)}=\mathcal{K}\,1\oplus \mathcal{K}\,1^{\#}$  is the 
2-dimensional  associative  algebra over $\mathcal{K}$,   $\{1, \,1^{\#}\}$ is a $\mathcal{K}$-basis 
for the  $\mathcal{K}$-space $\mathcal{K}^{(2)}$ and  the associative product on 
$\mathcal{K}^{(2)}$ is defined by
$$
(x_1+x_2\,1^{\#})(y_1+y_2\,1^{\#})=x_1y_1+(x_1y_2+x_2y_1)\,1^{\#},
$$
where $x_1$, $x_2$, $y_1$, $y_2\in \mathcal{K}$.  In particular,  $\mathcal{C}^{(2)}$ is called {\bf dual  complex algebra}, where
 $\mathcal{C}$ is the complex number field. 

If 
$x=x_1+x_2\,1^{\#}\in \mathcal{K}^{(2)}$ with $x_1$, $x_2\in \mathcal{K}$,
then $Re\,x:=x_1$ and $Ze\,x:=x_2$ are still called the {\bf  real part} and the  {\bf zero-divisor part} of $x$, respectively.

\medskip
Let $M_{n,m}(\mathcal{K})$ be the set of the ordinary $n\times m$ matrices whose entries are in a set $\mathcal{K}$. A matrix
$A=\left[\begin{array}{ccc}
A_1&|& A_21^{\#}\\
\hline
A_31^{\#}&|&A_41^{\#}
\end{array}\right]$ is called a
{\bf Hu-Liu $(s,t;n,m)$-matrix over $\mathcal{K}$ } if 
$$
A_1\in M_{s,n}(\mathcal{K}^{(2)}), \,\, A_2\in M_{s,m}(\mathcal{K}),\,\,
A_3\in M_{t,n}(\mathcal{K}),\,\, A_4\in M_{t,m}(\mathcal{K}),
$$
where $A_21^{\#}$, $A_31^{\#}$ and $A_41^{\#}$ are just the ordinary scalar multiplication for matrices. 
$A_1=(a_{ij})\in M_{s,n}(\mathcal{K}^{(2)})$ can be also expressed as 
$$A_1=Re\,A_1 +(Ze\, A_1)1^{\#},$$ where both $Re\,A_1=(Re\,a_{ij})$
and $Ze\,A_1=(Ze\,a_{ij})$ are the matrices in  $M_{s,n}(\mathcal{K})$.

Note that the Hu-Liu $(s,t;n,m)$-matrix $A=\left[\begin{array}{ccc}
A_1&|& A_21^{\#}\\
\hline
A_31^{\#}&|&A_41^{\#}
\end{array}\right]$ is also denoted by $A=\left[\begin{array}{lcl}
A_{11}&|& A_{12}\\
\hline
A_{21}&|&A_{22}
\end{array}\right]$, where
$$
A_{11}=A_1, \quad A_{12}=A_21^{\#},\quad A_{21}=A_31^{\#}, \quad A_{22}=A_41^{\#}
$$
and $A_{ij}$ is called the {\bf $(i, j)$-block of $A$} with $(i, j)\in\{(1, 1), (1, 2), (2, 1), (2, 2)\}$.

A  Hu-Liu $(n,m;n,m)$-matrix  is  called a {\bf square Hu-Liu $(n,m)$-matrix} . 
Let  $M_{s, t;n,m}^{\#}(\mathcal{K}^{(2)})$ be the set of  $(s,t;n,m)$-Hu-Liu matrices . For 
 $r\in \mathcal{K}^{(2)}$ and 
$$
A=\left[\begin{array}{ccc}
A_1&|& A_21^{\#}\\
\hline
A_31^{\#}&|&A_41^{\#}
\end{array}\right], \quad
B=\left[\begin{array}{ccc}
B_1&|& B_21^{\#}\\
\hline
B_31^{\#}&|&B_41^{\#}
\end{array}\right]
\in M_{s, t;n,m}^{\#}(\mathcal{K}^{(2)}),
$$ we define $A+B$ and $rA$ by
\begin{equation}\label{eq105}
A+B:=\left[\begin{array}{ccc} A_1+B_1&|&(A_2+B_2)1^{\#}\\\hline (A_3+B_3)1^{\#}&|&(A_4+B_4)1^{\#}\end{array}\right]
\end{equation}
and
\begin{equation}\label{eq106}
rA:==\left[\begin{array}{ccc} rA_1&|&rA_21^{\#}\\\hline rA_31^{\#}&|&rA_41^{\#}\end{array}\right]
=\left[\begin{array}{ccc} rA_1&|&(Re\,r)A_21^{\#}\\\hline (Re\,r)A_31^{\#}&|&(Re\,r)A_41^{\#}\end{array}\right].
\end{equation}
Using the addition defined by (\ref{eq105}) and the $\mathcal{K}^{(2)}$-module 
action defined by (\ref{eq106}), $M_{s, t;n,m}^{\#}(\mathcal{K}^{(2)})$ becomes a $\mathcal{K}^{(2)}$-module. We call
$M_{s, t;n,m}^{\#}$ the 
{\bf general Hu-Liu $(s, t;n,m)$-matrix $\mathcal{K}^{(2)}$-module.} 

\bigskip
Our generalization of the ordinary matrix product is given in the following

\medskip
\begin{definition}\label{def1.6} Let 
$B=\left[\begin{array}{ccc} (B_1)_{p\times s}&|&(B_21^{\#})_{p\times t}\\\hline 
(B_31^{\#})_{q\times s}&|&(B_41^{\#})_{q\times t}\end{array}\right]$ be a Hu-Liu $(p,q;s,t)$-matrix,  and let
$A=\left[\begin{array}{ccc} (A_1)_{s\times n}&|&(A_21^{\#})_{s\times m}\\\hline 
(A_31^{\#})_{t\times n}&|&(A_41^{\#})_{t\times m}\end{array}\right]$ be a Hu-Liu $(s,t;n,m)$-matrix. We define the {\bf Hu-Liu matrix product $B\# A$} of $B$ and $A$ to be the Hu-Liu $(p,q;n,m)$-matrix $BA$ given by
\begin{eqnarray}\label{eq107}
B\#A&=&\left[\begin{array}{ccc} (B_1)_{p\times s}&|&(B_21^{\#})_{p\times t}\\\hline 
(B_31^{\#})_{q\times s}&|&(B_41^{\#})_{q\times t}\end{array}\right]
\left[\begin{array}{ccc} (A_1)_{s\times n}&|&(A_21^{\#})_{s\times m}\\\hline 
(A_31^{\#})_{t\times n}&|&(A_41^{\#})_{t\times m}\end{array}\right]\nonumber\\
&&\nonumber\\
&=&
\left[\begin{array}{ccc} B_1A_1+B_2A_31^{\#}&|&\Big((Re\,B_1)A_2+B_2A_4\Big)1^{\#}\\\hline 
\Big(B_3(Re\, A_1)+B_4A_3\Big)1^{\#}&|&(B_4A_4)1^{\#}\end{array}\right].
\end{eqnarray}
\end{definition}

\medskip
Like the ordinary matrix product, the Hu-Liu matrix product is also associative.

\medskip
\medskip
\begin{proposition}\label{pr1.6} Let $C=\left[\begin{array}{ccc} (C_1)_{h\times p}&|&(C_21^{\#})_{h\times q}\\\hline 
(C_31^{\#})_{k\times p}&|&(C_41^{\#})_{k\times q}\end{array}\right]$ be a Hu-Liu $(h,k;p,q)$-matrix,  and let $B$ and $A$ be the Hu-Liu
matrices in Definition \ref{def1.6}. Then 
\begin{equation}\label{eq108}
C\#(B\#A)=(C\#B)\#A.
\end{equation}
\end{proposition}

\medskip
\noindent
{\bf Proof} By (\ref{eq107}), we have
\begin{equation}\label{eq109}
C\#B=
\left[\begin{array}{ccc} C_1B_1+C_2B_31^{\#}&|&\Big((Re\,C_1)B_2+C_2B_4\Big)1^{\#}\\\hline 
\Big(C_3(Re\, B_1)+C_4B_3\Big)1^{\#}&|&(C_4B_4)1^{\#}\end{array}\right].
\end{equation}

It follows from  (\ref{eq107}) and  (\ref{eq109}) that
\begin{eqnarray}\label{eq110}
&&\Big(C\#(B\#A)\Big)_{11}=C_1(B_1A_1+B_2A_31^{\#})+C_2(B_3(Re\, A_1)+B_4A_3\Big)1^{\#}\nonumber\\
&=&C_1B_1A_1+C_1B_2A_31^{\#}+C_2B_3(Re\, A_1)1^{\#}+C_2B_4A_31^{\#}\nonumber\\
&=&C_1B_1A_1+(Re\,C_1)B_2A_31^{\#}+C_2B_3(Re\, A_1)1^{\#}+C_2B_4A_31^{\#}\nonumber\\
&=&C_1B_1A_1+(Re\,C_1)B_2A_31^{\#}+(C_2B_3)A_11^{\#}+C_2B_4A_31^{\#}\nonumber\\
&=&(C_1B_1+(C_2B_3)1^{\#})A_1+((Re\,C_1)B_2+C_2B_4)A_31^{\#}\nonumber\\
&=&\Big((C\#B)\#A\Big)_{11}.
\end{eqnarray}

\medskip
Similarly, we have
\begin{equation}\label{eq111}
\Big((C\#B)\#A\Big)_{ij}=\Big((C\#B)\#A\Big)_{ij}\quad\mbox{for $(i,j)\in\{(1,2), (2,1), (2,2)\}$}.
\end{equation}

By (\ref{eq110}) and (\ref{eq111}),   the equation (\ref{eq108}) holds.

\hfill\raisebox{1mm}{\framebox[2mm]{}}

\bigskip
Let  $M_{n,m}^{\#}(\mathcal{K}^{(2)}):=M_{n, m;n,m}^{\#}(\mathcal{K}^{(2)})$ be the set of  square Hu-Liu 
$(n,m)$-matrices. By Proposition \ref{pr1.6},  $M_{n,m}^{\#}(\mathcal{K}^{(2)})$ is an 
associative real algebra with the identity  $I^{\#}_{n,m}$, where
$$
I^{\#}_{n,m}:=\left[\begin{array}{ccc} I_n&|&0_{p\times t}\\\hline 
0_{q\times s}&|&I_m1^{\#}\end{array}\right]
$$
and $I_n$ is the ordinary identity matrix in $M_{n,n}(\mathcal{K})$.

\medskip
\begin{definition}\label{def1.7} Let $A$ be a square  Hu-Liu $(n,m)$-matrix. Then $A$ is 
{\bf invertible} if there exists a square  Hu-Liu $(n,m)$-matrix $B$ such that $A\#B=B\#A=I^{\#}_{n,m}$. 
The Hu-Liu $(n,m)$-matrix $B$ is called 
the {\bf inverse} of $A$ and is denoted by $A^{\stackrel{\#}{-}1}$.
\end{definition}

\medskip
\begin{proposition}\label{pr1.7} A square Hu-Liu $(n,m)$-matrix 
$$A=\left[\begin{array}{ccc} (A_1)_{n\times n}&|&(A_21^{\#})_{n\times m}\\\hline 
(A_31^{\#})_{m\times n}&|&(A_41^{\#})_{m\times m}\end{array}\right]$$ is invertible if and only if the $n\times n$ matrix $Re\,A_1$  and the $m\times m$ matrix $A_4$ are invertible, in which case, the inverse $A^{\stackrel{\#}{-}1}$ of $A$ is 
given by
\begin{equation}\label{eq114}
A^{\stackrel{\#}{-}1}=\left[\begin{array}{ccc}
\big(A^{\stackrel{\#}{-}1}\big)_1&|&
-(Re\,A_1)^{-1}A_2A_4^{-1}1^{\#}\\\hline 
-A_4^{-1}A_3(Re\,A_1)^{-1}1^{\#}&|&A_4^{-1}1^{\#}\end{array}\right],
\end{equation}
where 
\begin{eqnarray}\label{eq115}
\big(A^{\stackrel{\#}{-}1}\big)_1&=&(Re\,A_1)^{-1}-(Re\,A_1)^{-1}(Ze\,A_1)(Re\,A_1)^{-1}1^{\#}+\nonumber\\
&&\quad +(Re\,A_1)^{-1}A_2A_4^{-1}A_3(Re\,A_1)^{-1}1^{\#}.
\end{eqnarray}
\end{proposition}

\medskip
\noindent
{\bf Proof} A direct computation.

\hfill\raisebox{1mm}{\framebox[2mm]{}}

\bigskip
\section{ $R^{(2)}$-root systems}

Let $V$ be a  finite dimensional  $\mathcal{R}^{(2)}$-module $V$ with 
$\dim_{\mathcal{R}^{(2)}}(V)=(n,\,m)$. Let
$$
GL^{\#}_{n,m}(V, \mathcal{R}^{(2)}):=\{T\,|\, 
\mbox{$T:V\to V$ is an invertible $\mathcal{R}^{(2)}$-map}\}.
$$
Clearly, $GL^{\#}_{n,m}(V, \mathcal{R}^{(2)})$ is a group which is called the 
{\bf general $\mathcal{R}^{(2)}$-linear group} of all  invertible $\mathcal{R}^{(2)}$-maps of $V$.

\medskip
\begin{definition}\label{def2.1} Let $(V,  \kappa (\,, \,))$ be a real Hu-Liu inner product  
$\mathcal{R}^{(2)}$-module.  An invertible $\mathcal{R}^{(2)}$-map $T: V\to V$ is said to be 
{\bf $\mathcal{R}^{(2)}$-orthogonal map} of $V$ if
$$
\kappa (T(x), T(y))=\kappa (x, y)\quad\mbox{for all $x$, $y\in V$.}
$$
\end{definition}

\medskip
Let $O^{\#}_{n,m}(V, \mathcal{R}^{(2)})$, which is also denoted by  $O^{\#}_{n,m}(V)$ or
$O^{\#}(V)$, be the set  of 
$\mathcal{R}^{(2)}$-orthogonal maps of $V$. For $x, y\in V$. Clearly,
 $O^{\#}_{n,m}(V, \mathcal{R}^{(2)})$ is a subgroup of $GL^{\#}_{n,m}(V, \mathcal{R}^{(2)})$ and 
is called the {\bf $\mathcal{R}^{(2)}$-orthogonal group}. 

\medskip
\begin{definition}\label{def1} Let $W$ be $\mathcal{R}^{(2)}$-submodule of a $(n, m)$-dimensional $\mathcal{R}^{(2)}$-module $V$. We say that
$$
\mbox{$W$ is a}\left\{\begin{array}{l}\mbox{{\bf type 1 $\mathcal{R}^{(2)}$-hyperplane in $V$} if  $\dim_{\mathcal{R}^{(2)}}(W)=(n-1,\,m)$,}\\
\mbox{{\bf type 2 $\mathcal{R}^{(2)}$-hyperplane in $V$} if  $\dim_{\mathcal{R}^{(2)}}(W)=(n,\,m-1)$.}
\end{array}\right.
$$
\end{definition}

\medskip
 Let $(V,  \kappa (\,, \,))$ be a real Hu-Liu inner product  
$\mathcal{R}^{(2)}$-module with $\dim_{\mathcal{R}^{(2)}}(V)=(n,\,m)$.  Clearly, we have
$$v\not \in Im\,1^{\#}_V\Longleftrightarrow   v\in  \big(V\setminus  Ker\,1^{\#}_V\big)\bigcup  
\big(Ker\,1^{\#}_V\setminus Im\,1^{\#}_V\big).$$
By Proposition \ref{pr1.4}, the $\mathcal{R}^{(2)}$-submodule  $\mathcal{R}^{(2)}v$  is pleasant for 
$v\not \in Im\,1^{\#}_V$. It follows from
this fact and Proposition \ref{pr1.5} that
\begin{equation}\label{eq131}
V=\mathcal{R}^{(2)}v\oplus \big(\mathcal{R}^{(2)}v\big)_{\perp}\quad\mbox{ for 
$v\not \in Im\,1^{\#}_V$},
 \end{equation}
where $ \big(\mathcal{R}^{(2)}v\big)_{\perp}$ is the  $\mathcal{R}^{(2)}$-orthogonal complement of 
$\mathcal{R}^{(2)}v$.  It is easy to check that $\big(\mathcal{R}^{(2)}v\big)_{\perp}$ is a  type 1 hyperplane for
$v\in V\setminus  Ker\,1^{\#}_V$, and $\big(\mathcal{R}^{(2)}v\big)_{\perp}$ is a  type 2 hyperplane for
$Ker\,1^{\#}_V\setminus Im\,1^{\#}_V$.

\medskip
We now introduce the counterparts of the reflections in a real vector spaces.

\medskip
\begin{definition}\label{def2.3}  Let $(V,  \kappa (\,, \,))$ be a finite dimensional real Hu-Liu inner product  
$\mathcal{R}^{(2)}$-module. For $v\not\in  Im\,1^{\#}_V$, the
{\bf $\mathcal{R}^{(2)}$-reflection  } $S_v$ determined by $v$ is defined by
\begin{equation}\label{eq132}
S_v(x):=\left\{\begin{array}{ll}x-2\,\displaystyle\frac{\kappa (x, v)}{\kappa (v, v)}v& 
\mbox{if $v\in V\setminus  Ker\,1^{\#}_V$},\\&\\
x-2\,\displaystyle\frac{Ze\,\kappa (x, v)}{Ze\,\kappa (v, v)}v&\mbox{if $v\in Ker\,1^{\#}_V\setminus Im\,1^{\#}_V$}.
\end{array}\right.
\end{equation}
We say that
$$
\mbox{$S_v$ is a}\left\{\begin{array}{l}\mbox{{\bf type 1 $\mathcal{R}^{(2)}$-reflection} if  
$v\in V\setminus  Ker\,1^{\#}_V$}\\
\mbox{{\bf type 2 $\mathcal{R}^{(2)}$-reflection} if  $v\in Ker\,1^{\#}_V\setminus Im\,1^{\#}_V$.}
\end{array}\right.
$$
\end{definition}

\medskip
Clearly,   the $\mathcal{R}^{(2)}$-reflection $S_v$ sends $v$ to $-v$ while fixing pointwise the $\mathcal{R}^{(2)}$-hyperplane 
$\big(\mathcal{R}^{(2)}v\big)_{\perp}$.  Here are some basic properties of $\mathcal{R}^{(2)}$-reflections.

\bigskip
Every $\mathcal{R}^{(2)}$-reflection of $V$ is a $\mathcal{R}^{(2)}$-orthogonal map of $V$. A finite subgroup of the $\mathcal{R}^{(2)}$-orthogonal group $O^{\#}_{n,m}(V)$ generated by  $\mathcal{R}^{(2)}$-reflections is called a finite  {\bf $\mathcal{R}^{(2)}$-reflection group} on $V$. In the following of this paper, we use $W^{\#}$ or $W^{\#}(V)$ to denote a finite  
 $\mathcal{R}^{(2)}$-reflection group on $V$.

\bigskip
We now introduce the concept of $\mathcal{R}^{(2)}$-root system which generalizes the concept of the root system in Euclidean spaces.

\medskip
\begin{definition}\label{def2.4} Let $V$ be a finite dimensional  
$\mathcal{R}^{(2)}$-module. A non-empty finite subset $\Phi$ of $V$ is called a 
{\bf $\mathcal{R}^{(2)}$-root system} if the set $\Phi$ satisfies the following conditions:
\begin{description}
\item[(i)] There exist the subsets $\Phi_{Re}$ and  $\Phi_{Ze}$ of $V$ such that
$$
\Phi=\Phi_{Re}\cup \Phi_{Ze}, \quad \Phi_{Re}\subseteq V\setminus  Ker\,1^{\#}_V, \quad
 \Phi_{Ze}\subseteq Ker\,1^{\#}_V\setminus Im\,1^{\#}_V;
$$
\item[(ii)]  $\Phi_{\heartsuit}\cap \mathcal{R}^{(2)}v_{\heartsuit}=\{v_{\heartsuit}, \, -v_{\heartsuit}\}$ for  
$v_{\heartsuit}\in \Phi_{\heartsuit}$ with 
$\heartsuit\in \{Re, \, Ze\}$;
\item[(iii)] $S_{v_{\heartsuit}}(\Phi_{\diamondsuit})=\Phi_{\diamondsuit}$ for  $v_{\heartsuit}\in \Phi_{\heartsuit}$ with 
$\heartsuit, \diamondsuit\in \{Re, \, Ze\}$.
\end{description}
\end{definition}

\medskip
We define a binary relation $<$ on  the  dual real number algebra $\mathcal{R}^{(2)}$ by
\begin{equation}\label{eq148}
x^{(2)}<y^{(2)} \Longleftrightarrow \mbox{either } Re\,x^{(2)}<Re\,y^{(2)} \,\mbox{or} \,
\left\{\begin{array}{l}Re\,x^{(2)}=Re\,y^{(2)},\\Ze\,x^{(2)}<Ze\,y^{(2)}.
\end{array}\right.,
 \end{equation}
where $x^{(2)}$, $y^{(2)}\in \mathcal{R}^{(2)}$.
The basic properties of the  binary relation $<$ on   $\mathcal{R}^{(2)}$ are given in the following proposition.

\medskip
\begin{proposition}\label{pr2.4} If $<$ is the  binary relation  defined by  (\ref{eq148}), then
\item[(i)] For all $u$, $v$, $w\in \mathcal{R}^{(2)}$, if $u\le v$ and $v\le w$, then $u\le w$.
\item[(ii)] For each $u$, $v\in \mathcal{R}^{(2)}$ exactly one of $u<v$, $u=v$, $v<u$ holds.
\item[(iii)] $u<v\Longrightarrow u+w<v+w$ for all $u$, $v$, $w\in \mathcal{R}^{(2)}$.
\item[(iv)] If  $0<u\in \mathcal{R}^{(2)}$ and $0<c^{(2)}\in \mathcal{R}^{(2)}$,  then $c^{(2)}u\ge  0$, where the equality holds if and only if 
 $\{c^{(2)}, u\}\subseteq \mathcal{R}1^{\#}$.
\item[(v)] If  $0<u\in \mathcal{R}^{(2)}$ and $0>c^{(2)}\in \mathcal{R}^{(2)}$,  then $c^{(2)}u\le  0$, where the equality holds if and only if 
 $\{c^{(2)}, u\}\subseteq \mathcal{R}1^{\#}$.
\end{proposition}

\medskip
\noindent
{\bf Proof}   (i), (ii) and (iii) are clear. The proof of (iv) and (v) are similar. For example, let us prove (iv).
First we prove
\begin{equation}\label{eq149}
c^{(2)}u\ge  0.
\end{equation}

\medskip
Since
\begin{eqnarray*}
c^{(2)}u&=&(Re\,c^{(2)}+1^{\#}Ze\,c^{(2)})(Re\,u+1^{\#}Ze\,u)\\
&=&Re\,c^{(2)}Re\,u+1^{\#}(Re\,c^{(2)}\, Ze\,u+Ze\,c^{(2)}\, Re\,u),
\end{eqnarray*}
we get
\begin{equation}\label{eq150}
Re\,(c^{(2)}u)=Re\,c^{(2)}Re\,u\quad\mbox{and}\quad 
Ze\,(c^{(2)}u)=Re\,c^{(2)}\, Ze\,u+Ze\,c^{(2)}\, Re\,u.
\end{equation}

If $Re\,u>0$ and $Re\,c^{(2)}>0$, then 
$$Re\,(c^{(2)}u)\stackrel{ (\ref{eq150})}{=} 
Re\,c^{(2)}Re\,u>0\stackrel{ (\ref{eq148})}{\Longrightarrow}  \mbox{(\ref{eq149}) holds.}$$

If $Re\,u>0$ and $\left\{\begin{array}{l}Re\,c^{(2)}=0,\\Ze\,c^{(2)}>0\end{array}\right.$, then 
$$
Re\,(c^{(2)}u)=0\,\mbox{and}\,
Ze\,(c^{(2)}u)\stackrel{ (\ref{eq150})}{=} Ze\,c^{(2)}\, Re\,u
>0\stackrel{ (\ref{eq148})}{\Longrightarrow}  \mbox{(\ref{eq149}) holds.}$$

If  $\left\{\begin{array}{l}Re\,u=0,\\Ze\,u>0\end{array}\right.$ and $Re\,c^{(2)}>0$, then 
$$
Re\,(c^{(2)}u)=0\,\mbox{and}\,
Ze\,(c^{(2)}u)\stackrel{ (\ref{eq150})}{=}Re\,c^{(2)}\, Ze\,u
>0\stackrel{ (\ref{eq148})}{\Longrightarrow}  \mbox{(\ref{eq149}) holds.}$$

If  $\left\{\begin{array}{l}Re\,u=0,\\Ze\,u>0\end{array}\right.$ and 
$\left\{\begin{array}{l}Re\,c^{(2)}=0,\\Ze\,c^{(2)}>0\end{array}\right.$, then 
$$
Re\,(c^{(2)}u)=0\,\mbox{and}\,
Ze\,(c^{(2)}u)\stackrel{ (\ref{eq150})}{=}0\stackrel{ (\ref{eq148})}{\Longrightarrow}  \mbox{(\ref{eq149}) holds.}$$

This completes the proof of (\ref{eq149}) foe all cases.

\medskip
If $c^{(2)}u=0$, then
\begin{equation}\label{eq152}
 Re\,c^{(2)}Re\,u= 0
\end{equation}
and
\begin{equation}\label{eq153}
  Re\,c^{(2)}\, Ze\,u+Ze\,c^{(2)}\, Re\,u= 0.
\end{equation}

\medskip
By (\ref{eq152}), we have
$$
\mbox{either} \left\{\begin{array}{c}Re\,c^{(2)}\ne 0,\\Re\,u=0; 
\end{array}\right. \mbox{or}\,  \left\{\begin{array}{c}Re\,c^{(2)}= 0,\\Re\,u\ne 0; 
\end{array}\right.  \mbox{or}\,  \left\{\begin{array}{c}Re\,c^{(2)}= 0,\\Re\,u= 0.
\end{array}\right. 
$$

\medskip
If $Re\,c^{(2)}\ne 0$ and $Re\,u=0$, then $ Ze\,u=0$ by  (\ref{eq153}) and  $u=Re\,u+1^{\#}Ze\,u=0$, which is impossible because 
$u>0$. If $Re\,c^{(2)}= 0$ and $Re\,u\ne 0$, then $ Ze\,c^{(2)}=0$ by  (\ref{eq153}) and  $c^{(2)}=Re\,c^{(2)}+1^{\#}Ze\,c^{(2)}=0$, which is impossible because $c^{(2)}>0$. This proves that we have to have 
$Re\,c^{(2)}=0$ and $Re\,u=0$, which implies   
$$c=Re\,c^{(2)}+1^{\#}Ze\,c^{(2)}=1^{\#}Ze\,c^{(2)}\in  \mathcal{R}1^{\#}$$
 and 
$$u=Re\,u+1^{\#}Ze\,u=1^{\#}Ze\,u\in \mathcal{R}1^{\#}.$$
Hence, the property (iv) in Proposition \ref{def2.4} holds.

\hfill\raisebox{1mm}{\framebox[2mm]{}}

\bigskip
We now introduce the following generalization of  the ordinary total ordering on the real vector spaces
in real reflection groups.

\medskip
\begin{definition}\label{def2.5} A {\bf total $\mathcal{R}^{(2)}$-ordering} of a $\mathcal{R}^{(2)}$-module $V$ is a binary relation 
on $V$ (denoted by $<$) satisfying the following axioms.
\begin{description}
\item[(i)] For all $u$, $v$, $w\in V$, if $u\le v$ and $v\le w$, then $u\le w$.
\item[(ii)] For each $u$, $v\in V$ exactly one of $u<v$, $u=v$, $v<u$ holds.
\item[(iii)] $u<v\Longrightarrow u+w<v+w$ for all $u$, $v$, $w\in V$.
\item[(iv)] If  $0<u\in V$ and $c^{(2)}\in \mathcal{R}^{(2)}$,  then $c^{(2)}u\ge 0$ if $c^{(2)}>0$, while 
$c^{(2)}u\le 0$ if $c^{(2)}<0$.
\item[(v)] For  all $u$, $v\in V$,   if  $0<u\notin  Im\,1^{\#}_V$ and $0<v\notin  Im\,1^{\#}_V$ ,  then 
$0<u+v\notin  Im\,1^{\#}_V$.
\end{description}
We say that $u\in V$ is {\bf positive} if $u>0$, and $u\in V$ is {\bf negative} if $u<0$.
\end{definition}

\medskip
By Proposition \ref{def2.4},  the binary relation $<$ defined by  (\ref{eq148}) is a total $\mathcal{R}^{(2)}$-ordering of the 
$\mathcal{R}^{(2)}$-module $\mathcal{R}^{(2)}$.
Let $V$ be a finite dimensional $\mathcal{R}^{(2)}$-module $V$ with an ordered $(\mathcal{R}^{(2)}, \mathcal{R})$-basis:
$(v_1^{(2)},  v_2^{(2)}, \dots, v_n^{(2)}\,\|\, v_{n+1}, v_{n+2},  \dots, v_{n+m} ).$ Clearly, the set
\begin{equation}\label{eq154}
\mathcal{B}=\{u_1, u_2, \dots, u_n, u_{n+1}, u_{n+2},  \dots, u_{n+m}, u_{n+m+1}, u_{n+m+2}, \dots, u_{2n+m}\}
\end{equation}
is an ordered $\mathcal{R}$-basis of the real vector space $V$, where 
$$
u_i:=\left\{\begin{array}{ll} v_i^{(2)} & \mbox{if $1\le i\le n$},\\ v_{i}& 
\mbox{if $n+1\le i\le n+m$},\\1^{\#}v_{i-n-m}^{(2)}&\mbox{if $n+m+1\le i\le 2n+m$}.
\end{array}\right.
$$
Clearly, we have $1^{\#}u_i=u_{n+m+i}$ for $1\le i\le n$ and
\begin{equation}\label{eq155}
Im\,1^{\#}_V=\displaystyle\sum_{i=n+m+1}^{2n+m}\, \mathcal{R}u_i\quad\mbox{and}\quad 
Ker\,1^{\#}_V=\displaystyle\sum_{i=n+1}^{2n+m}\, \mathcal{R}u_i.
\end{equation}

Using the the ordered  $\mathcal{R}$-basis $\mathcal{B}=\{u_i\}_{i=1}^{2n+m}$ of  the real vector space $V$, we get an ordinary 
total ordering $< $ on the real vector space $V$. In other words, if 
$x=\displaystyle\sum_{i=1}^{2n+m}\,x_iu_i$ and $y=\displaystyle\sum_{i=1}^{2n+m}\,y_iu_i$ 
with $x_i$, $y_i\in \mathcal{R}$ for $1\le i\le 2n+m$,  then 
\begin{eqnarray}\label{eq156}
&&\mbox{$x<y$ means that $x_k<y_k$,
where $1\le k\le 2n+m$}\nonumber\\
&&\mbox{and $k$ is the least index $i$ for which $x_i\ne y_i$.}
\end{eqnarray}

\medskip
\begin{proposition}\label{pr2.5} The binary relation defined by  (\ref{eq156}) is a  total $\mathcal{R}^{(2)}$-ordering of
a finite dimensional $\mathcal{R}^{(2)}$-module $V$.
\end{proposition}

\medskip
\noindent
{\bf Proof} Since  the binary relation $<$ defined by  (\ref{eq155}) is an ordinary total ordering  on the real vector space $V$,
Definition \ref{def2.5} (i), (ii) and (iii) hold. One can prove Both (iv) and (v)  by cases.

\hfill\raisebox{1mm}{\framebox[2mm]{}}

\bigskip
{\bf Remark} If $V$ be a finite dimensional $\mathcal{R}^{(2)}$-module,  $u\in V$ and  $c^{(2)}\in \mathcal{R}^{(2)}$, then
\begin{equation}\label{eq160}
 \mbox{    $c^{(2)}u=0$ and $Re\,c^{(2)}\ne 0$ $\Longrightarrow$  $u=0$}
\end{equation}
and
\begin{equation}\label{eq161}
 \mbox{  $c^{(2)}u=0$ and  $\left\{\begin{array}{l}Re\,c^{(2)}=0,\\Ze\,c^{(2)}\ne 0,\end{array}\right.$ $\Longrightarrow$  
$u\in Ker 1^{\#}_{V}$.}
\end{equation}
In fact,  if $Re\,c^{(2)}\ne 0$, then the inverse $Re\,c^{(2)})^{-1}$ exists and 
$$u=Re\,c^{(2)})^{-1}\cdot c^{(2)}u=0.$$
If  $\left\{\begin{array}{l}Re\,c^{(2)}=0,\\Ze\,c^{(2)}\ne 0,\end{array}\right.$, then
$
1^{\#}(Ze\,c^{(2)})u=c^{(2)}u=0\Longrightarrow 1^{\#}u=0\Longrightarrow
u\in Ker 1^{\#}_{V}.
$

\bigskip
Let $\Phi=\Phi_{Re}\cup \Phi_{Ze}$ be a $\mathcal{R}^{(2)}$-root system in a finite dimensional  
$\mathcal{R}^{(2)}$-module $V$. 
The elements of $\Phi$, $\Phi_{Re}$ and $\Phi_{Ze}$  are called {\bf $\mathcal{R}^{(2)}$-roots},  
{\bf real $\mathcal{R}^{(2)}$-roots} and  {\bf zero divisor $\mathcal{R}^{(2)}$-roots}, respectively. Let 
$$\Pi:=\{\alpha \, |\, \mbox{$\alpha\in \Phi$ and $\alpha $ is positive relative to some total $\mathcal{R}^{(2)}$-ordering of $V$}\}.$$
Then $\Pi$ is called the {\bf positive $\mathcal{R}^{(2)}$-system} in $\Phi$. By Definition \ref{def2.4} (ii), we have
$$
\Phi=\Pi\cup (-\Pi), \quad \Pi\cap (-\Pi)=\emptyset,
$$
where $-\Pi:=\{-\alpha \,|\, \alpha \in \Pi\}$ and $-\Pi$ is called the  {\bf negative $\mathcal{R}^{(2)}$-system} in $\Phi$.  

\medskip
\begin{definition}\label{def2.6} 
A subset $\Delta$ of  a $\mathcal{R}^{(2)}$-root system $\Phi$ in a finite dimensional  
$\mathcal{R}^{(2)}$-module $V$ is called a 
{\bf simple $\mathcal{R}^{(2)}$-system} if the subset $\Delta$ has the following two properties:
\begin{description}
\item[(i)] $\Delta$  is a $(\mathcal{R}^{(2)}, \mathcal{R})$-basis of the $\mathcal{R}^{(2)}$-module $span^{\#}(\Phi)$, where
 $span^{\#}(\Phi)$, which is called  the  {\bf $(\mathcal{R}^{(2)}, \mathcal{R})$-span} of $\Phi$, consists of  all  
$(\mathcal{R}^{(2)}\,\|\, \mathcal{R})$-linearly combination of  the  dual real vectors of $(\Phi_{Re}\,\|\, \Phi_{Ze})$;
\item[(ii)]   Each $\alpha\in \Phi$ is a $(\mathcal{R}^{(2)}\,\|\, \mathcal{R})$-linearly combination of   $(\Delta\cap\Phi_{Re}\,\|\, \Delta\cap\Phi_{Ze})$ with coefficients all of the same sign (all nonnegative or all nonpositive).
\end{description}
\end{definition}

\medskip
The following proposition proves that simple $\mathcal{R}^{(2)}$-systems exist.

\medskip
\begin{proposition}\label{pr2.6} If $\Phi=\Phi_{Re}\cup \Phi_{Ze}$ is a $\mathcal{R}^{(2)}$-root system in a finite dimensional  
$\mathcal{R}^{(2)}$-module $V$,  and  $\Pi$ is a  positive $\mathcal{R}^{(2)}$-system in $\Phi$, then there exists a unique 
simple $\mathcal{R}^{(2)}$-system  $\Delta$ such that  $\Delta\subseteq \Pi$.
\end{proposition}

\bigskip
\section{Real Hilbert-like Modules}

By \cite{KL2}, the dual real number algebra $\mathcal{R}^{(2)}$ is a normed algebra with respect to the 
norm  $||\,\,||$, where  the norm $||\,\,||$ is defined by 
$$
||x||:=\sqrt{2\,(Re\,x)^2+(Ze\,x)^2}\quad\mbox{for $x\in\mathcal{R}^{(2)}$.}
$$

\medskip
\begin{definition}\label{def3.1} A real vector space $V$ is called a {\bf real Hilbert-like module} 
if it  has the following three properties:
\begin{description}
\item[(i)] $V$ is a real Banach space equipped a norm  $||\,\,||$.
\item[(ii)] $V$ is a real Hu-Liu inner product  $\mathcal{R}^{(2)}$-module equipped  a real Hu-Liu inner 
product $ \kappa(\, ,\,)$.
\item[(iii)]  the norm $||\,\,||$ and the real Hu-Liu inner product $ \kappa(\, ,\,)$ satisfy the 
following inequality:
\begin{equation}\label{eq235}
||\kappa(x, \, y)||\le ||x||\,||y||\quad\mbox{for all $x$, $y\in V$.}
\end{equation}
\end{description}
A real Hilbert-like module $V$ is also denoted by $(V, \, ||\,\,||,\, \kappa(\, , \,))$ to indicate the norm $||\,\,||$ and 
the real Hu-Liu inner product  $\kappa(\, , \,)$ which define the real Hilbert-like module structure on $V$.
\end{definition}

\bigskip
To give the first example of real Hilbert-like modules, we need the following generalization of
Cauchy-Schwartz inequality.

\medskip
\begin{proposition}\label{pr3.1} If $\{x_{i1}, \,y_{i1}, \, x_{i2}, \, y_{i2}\}_{i=1}^n$ and $\{x_{n+j}, \, y_{n+j}\}_{j=1}^m$ are the subsets of 
real number field $\mathcal{R}$, then the following {\bf Hu-Liu inequality} holds:
$$
2\Big(\displaystyle\sum_{i=1}^n x_{i1}y_{i1}\Big)^2+
\Big(\displaystyle\sum_{i=1}^n(x_{i1}y_{i2}+x_{i2}y_{i1})+
\displaystyle\sum_{j=1}^m\,x_{n+j}y_{n+j}\Big)^2
$$
$$
\le\Big[2\displaystyle\sum_{i=1}^n(x_{i1})^2+\displaystyle\sum_{i=1}^n(x_{i2})^2+
\displaystyle\sum_{j=1}^m\,(x_{n+j})^2\Big]
\Big[2\displaystyle\sum_{i=1}^n(y_{i1})^2+\displaystyle\sum_{i=1}^n(y_{i2})^2+\displaystyle\sum_{j=1}^m\,(y_{n+j})^2\Big].
$$
\end{proposition}

\medskip
\noindent
{\bf Proof}  It is a straightforward computation.  See \cite{KL} for details.

\hfill\raisebox{1mm}{\framebox[2mm]{}}

\medskip
A real Hilbert-like module $V$ is called {\bf finite dimensional} if $V$ is a finite dimensional $\mathcal{R}^{(2)}$-module. A
real Hilbert-like module $V$ that is not finite dimensional is called {\bf infinite dimensional}.
There are two basic examples of real Hilbert-like modules. The first one is a  finite dimensional real Hilbert-like module, which is given in the following

\medskip
\begin{proposition}\label{pr3.2}  The real $(n, m)$-dimensional $\mathcal{R}^{(2)}$-module $\mathcal{R}^{(2)n,m}$ is a 
real Hilbert-like module with respect to the norm  $||\,\,||_{n,m}$ and the standard real Hu-Liu inner 
product $ \kappa_{n,m}(\, ,\,)$, where  $ \kappa_{n,m}(\, ,\,)$ is defined by (\ref{eq1}) and $||\,\,||_{n,m}$ is defined by
\begin{equation}\label{eq236}
||x||_{n,m}:=\Big(2\displaystyle\sum_{i=1}^n (x_{i1})^2+
\displaystyle\sum_{i=1}^n (x_{i2})^2+\displaystyle\sum_{j=1}^m (x_{n+j})^2\Big)^{\frac12}
\end{equation}
for $
x=\displaystyle\sum_{i=1}^n (x_{i1}+1^{\#}x_{i2})e_i^{(2)}+\sum_{j=1}^m\,x_{n+j}e_{n+j}\in \mathcal{R}^{(2)n,m},
$
where  $x_{i1}$, $x_{i2}$, $x_{n+j}\in \mathcal{R}$ with
$1\le i\le n$ and $1\le j\le m$ and  $(\{e_i^{(2)}\}_{i=1}^n\,\|\, \{e_{n+j}\}_{j=1}^m)$ is   the 
standard $(\mathcal{R}^{(2)}, \mathcal{R})$-basis of $\mathcal{R}^{(2)n, m}$.
\end{proposition}

\medskip
\noindent
{\bf Proof} We define $<\, ,\,>_{n,m}: \mathcal{R}^{(2)n,m}\times \mathcal{R}^{(2)n,m}\to \mathcal{R}$ by
\begin{equation}\label{eq237}
<x, y>_{n,m}:=2\displaystyle\sum_{i=1}^n x_{i1}y_{i1}+\displaystyle\sum_{i=1}^n x_{i2}y_{i2}
+\displaystyle\sum_{j=1}^m x_{n+j}y_{n+j}
\end{equation}
for 
$$
x=\displaystyle\sum_{i=1}^n (x_{i1}+1^{\#}x_{i2})e_i^{(2)}+\sum_{j=1}^m\,x_{n+j}e_{n+j}\in \mathcal{R}^{(2)n,m},
$$
$$
y=\displaystyle\sum_{i=1}^n (y_{i1}+1^{\#}y_{i2})e_i^{(2)}+\sum_{j=1}^m\,y_{n+j}e_{n+j}\in \mathcal{R}^{(2)n,m},
$$
where  $x_{i1}$, $x_{i2}$, $y_{i1}$, $y_{i2}$, $x_{n+j}$,  $y_{n+j}\in \mathcal{R}$ with
$1\le i\le n$ and $1\le j\le m$ and  $(\{e_i^{(2)}\}_{i=1}^n\,\|\, \{e_{n+j}\}_{j=1}^m)$ is   the 
standard $(\mathcal{R}^{(2)}, \mathcal{R})$-basis of $\mathcal{R}^{(2)n, m}$.  Then $<\,, \,>_{n,m}$ is a real inner product
on $\mathcal{R}^{(2)n,m}$. 

Since $\mathcal{R}^{(2)n,m}$ is a $(2n+m)$-dimensional real vector space, the real inner product space  $\mathcal{R}^{(2)n,m}$ 
is a  real Banach space equipped the norm  $||\,\,||_{n,m}$, where
\begin{equation}\label{eq238}
||x||_{n,m}:=\sqrt{<x, x>_{n,m}}\stackrel{(\ref{eq237})}{=}\Big(2\displaystyle\sum_{i=1}^n (x_{i1})^2+
\displaystyle\sum_{i=1}^n (x_{i2})^2+\displaystyle\sum_{j=1}^m (x_{n+j})^2\Big)^{\frac12} 
\end{equation}
for $x=\displaystyle\sum_{i=1}^n (x_{i1}+1^{\#}x_{i2})e_i^{(2)}+\sum_{j=1}^m\,x_{n+j}e_{n+j}\in \mathcal{R}^{(2)n,m}$.  Hence, we have
$$
||<x, y>_{n,m}||^2\stackrel{(\ref{eq1})}{=}2\Big(\displaystyle\sum_{i=1}^n x_{i1}y_{i1}\Big)^2+
\Big(\displaystyle\sum_{i=1}^n(x_{i1}y_{i2}+x_{i2}y_{i1})+
\displaystyle\sum_{j=1}^m\,x_{n+j}y_{n+j}\Big)^2
$$
and
$$\Big(||x||_{n,m}\cdot||y||_{n,m}\Big)^2=(||x||_{n,m})^2\cdot (||y||_{n,m})^2$$
$$\stackrel{(\ref{eq238})}{=}\Big[2\displaystyle\sum_{i=1}^n(x_{i1})^2+\displaystyle\sum_{i=1}^n(x_{i2})^2+
\displaystyle\sum_{j=1}^m\,(x_{n+j})^2\Big]
\Big[2\displaystyle\sum_{i=1}^n(y_{i1})^2+\displaystyle\sum_{i=1}^n(y_{i2})^2+\displaystyle\sum_{j=1}^m\,(y_{n+j})^2\Big].$$
It follows from Hu-Liu inequality that 
$$||<x, y>_{n,m}||^2\le \Big(||x||_{n,m}\cdot||y||_{n,m}\Big)^2,$$
which proves that Definition \ref{def3.1} (iii) holds.

\hfill\raisebox{1mm}{\framebox[2mm]{}}

\bigskip
We now introduce the second basic example $\ell^{2\#}$ of real Hilbert-like modules, which is the counterpart of the Hilbert 
space $\ell^{2}$. For $\{x_i^{(2)}\}_{i=1}^{\infty}\subseteq \mathcal{R}^{(2)}$ and 
$\{x_i\}_{i=1}^{\infty}\subseteq \mathcal{R}$, we use $(x_i^{(2)}\,|\, x_i1^{\#})$ to denote 
$(x_1^{(2)}, x_2^{(2)}, \dots \,|\, x_11^{\#}, x_21^{\#}, \dots)$.
Let
$$
\ell^{2\#}:=\Big\{(x_i^{(2)}\,|\, x_i1^{\#})\,|\, \mbox{$\displaystyle\sum_{i=1}^{\infty} ||x_i^{(2)}||^2<\infty$ and 
$\displaystyle\sum_{i=1}^{\infty} |x_i|^2<\infty$} \Big\}.
$$
We define
\begin{equation}\label{eq239}
(x_i^{(2)}\,|\, x_i1^{\#})+(y_i^{(2)}\,|\, y_i1^{\#}):=(x_i^{(2)}+y_i^{(2)}\,|\, x_i1^{\#})+y_i1^{\#}
\end{equation}
and 
\begin{equation}\label{eq240}
r^{(2)}\,(x_i^{(2)}\,|\, x_i1^{\#}):=(r^{(2)}x_i^{(2)}\,|\, r^{(2)}x_i1^{\#})=(r^{(2)}x_i^{(2)}\,|\, (Re\,r^{(2)})x_i1^{\#})
\end{equation}
for $(x_i^{(2)}\,|\, x_i1^{\#}), \, (y_i^{(2)}\,|\, y_i1^{\#})\in \ell^{2\#}$ and $r^{(2)}\in \mathcal{R}^{(2)}$. Clearly, both
 (\ref{eq236}) and  (\ref{eq236}) are well-defined.
Hence, $\ell^{2\#}$ becomes a  $\mathcal{R}^{(2)}$-module under the  addition  defined by (\ref{eq239})  
 and the scalar multiplication  defined by (\ref{eq240}).

\medskip
\begin{proposition}\label{pr3.3}  The $\mathcal{R}^{(2)}$-module $\ell^{2\#}$ is a 
infinite dimensional real Hilbert-like module with respect to the  norm $||\, \,||_{_{2\#}}$ and the  real Hu-Liu inner 
product $ \kappa_{_{2\#}}(\, ,\,)$, where  $||\, \,||_{_{2\#}}$ and $\kappa_{_{2\#}}(\, ,\,)$ are defined by
\begin{equation}\label{eq241}
||x||_{_{2\#}}:=\Big(2\displaystyle\sum_{i=1}^{\infty} (x_{i1})^2+\displaystyle\sum_{i=1}^{\infty}  (x_{i2})^2
+\displaystyle\sum_{i=1}^{\infty}  (x_i)^2\Big)^{\frac12}
\end{equation}
and
\begin{equation}\label{eq242}
\kappa_{_{2\#}}(x,\, y):=
\displaystyle\sum_{i=1}^{\infty} x_{i1}y_{i1}+
1^{\#}\Big(\displaystyle\sum_{i=1}^{\infty}(x_{i1}y_{i2}+x_{i2}y_{i1})+\displaystyle\sum_{i=1}^{\infty}  x_iy_i\Big)
\end{equation}
for $x=(x_{i1}+1^{\#}x_{i2}\,|\, x_i1^{\#})\in \ell^{2\#}$ and  $y=(y_{i1}+1^{\#}y_{i2}\,|\, y_i1^{\#})\in \ell^{2\#}$ with
$\{x_{i1}, x_{i2}, x_i, y_{i1}, y_{i2}, y_i\}_{i=1}^{\infty}\subseteq  \mathcal{R}$.
\end{proposition}

\medskip
\noindent
{\bf Proof} We define $<\,,\,>_{_{2\#}}:  \ell^{2\#}\times \ell^{2\#}\to \mathcal{R}$ by
\begin{equation}\label{eq243}
<x,\, y>_{_{2\#}}:=2\displaystyle\sum_{i=1}^{\infty} x_{i1} y_{i1}+\displaystyle\sum_{i=1}^{\infty}   x_{i2}y_{i2}
+\displaystyle\sum_{i=1}^{\infty}  x_iy_i
\end{equation}
for $x=(x_{i1}+1^{\#}x_{i2}\,|\, x_i1^{\#})\in \ell^{2\#}$ and  $y=(y_{i1}+1^{\#}y_{i2}\,|\, y_i1^{\#})\in \ell^{2\#}$ with
$\{x_{i1}, x_{i2}, x_i, y_{i1}, y_{i2}, y_i\}_{i=1}^{\infty}\subseteq  \mathcal{R}$. 

\medskip
One can check that both (\ref{eq242}) and (\ref{eq243}) are well-defined, $\ell^{2\#}$ is a Hilbert space equipped 
the norm  $||\,\,||_{_{2\#}}$ given by (\ref{eq241}), and $\ell^{2\#}$ is a real Hu-Liu inner product  
$\mathcal{R}^{(2)}$-module equipped  a real Hu-Liu inner product $\kappa_{_{2\#}}(\,,\, \,)$  given by (\ref{eq242}). 
Finally, it follows from  Hu-Liu inequality that
$$
||\kappa_{_{2\#}}(x,\, y)||\le ||x||_{_{2\#}}\cdot ||y||_{_{2\#}}.
$$

This completes the proof of Proposition \ref{pr3.3}.

\hfill\raisebox{1mm}{\framebox[2mm]{}}

\bigskip
The Riesz representation theorem in Hilbert spaces has a natural counterpart in Hilbert-like modules. For
convenience, we give  the counterpart of the Riesz representation theorem for 
finite dimensional Hilbert-like modules in the following

\medskip
\begin{proposition}\label{pr3.4}  Let $(V, \, ||\,\,||,\, \kappa(\, , \,))$ be a $(n, m)$-dimensional  Hilbert-like module. If 
$\varphi: V\to \mathcal{R}^{(2)}$ is a $\mathcal{R}^{(2)}$-module map, there exists a unique element $\xi\in V$ such 
that $\varphi (x)=\kappa(x , \xi)$ for all $x\in V$.
\end{proposition}

\medskip
\noindent
{\bf Proof} Although the proof of
the Riesz representation theorem for finite dimensional Hilbert spaces is short, the proof of  the counterpart 
of the Riesz representation theorem  for finite dimensional Hilbert-like modules is long. 

\medskip
\underline{Uniqueness:} Suppose that $\kappa(x , \xi)=\kappa(x , \xi')$ for all $x\in V$, where $xi$, $xi'\in  V$. Then
\begin{equation}\label{eq244}
\kappa(x , \xi-\xi')=0\quad  \mbox{for all $x\in V$.} 
\end{equation}
Hence, $\kappa(\xi-\xi' , \xi-\xi')=0$, which implies that  $Ze\,\kappa(\xi-\xi' , \xi-\xi')=0$. This fact and 
Definition \ref{def1.2} (v) give $\xi-\xi'\in Im\,1^{\#}_V$ or $\xi-\xi'=1^{\#}v$ for some $v\in V$. It follows that
\begin{eqnarray*}
&& 1^{\#}\kappa(v , v)=\kappa(v , \,1^{\#}v)=\kappa(v , \,\xi-\xi')\stackrel{(\ref{eq244})}{=}0\\
&\Longrightarrow& Re\,\kappa(v , v)=0\Longrightarrow v\in  Ker\,1^{\#}_V\\
&\Longrightarrow& \xi-\xi'=1^{\#}v=0\quad\mbox{or}\quad \xi=\xi'.
\end{eqnarray*}

\medskip
 \underline{Existence:} If $\varphi=0$,  we take $\xi=0$. Therefore, we can assume 
$\varphi\ne 0$.
Since $Im\,\varphi:=\{\varphi (v) \,|\, v\in V\}$ is a submodule of the $\mathcal{R}^{(2)}$-module
$\mathcal{R}^{(2)}$, we have either $Im\,\varphi=\mathcal{R}1^{\#}$ or $Im\,\varphi=\mathcal{R}^{(2)}$. In each of these two cases, Proposition \ref{pr3.4}
follows from
using the  Hu-Liu extension of the Gram-Schmidt 
process a few times.

\hfill\raisebox{1mm}{\framebox[2mm]{}}

\bigskip
\section{Zero divisor representation of finite group}

Let $V$ be a $(n, m)$-dimensional $\mathcal{C}^{(2)}$-module, where  $\mathcal{C}^{(2)}$ is the dual complex number algebra. Let
$$
GL^{\#}_{n,m}(V):=\{ T\,|\, \mbox{$T: V\to V$ is a $\mathcal{C}^{(2)}$ isomorphism}\}
$$
and
$$
GL^{\#}_{n,m}(\mathcal{C}^{(2)}):=\{ X\,|\, \mbox{$X\in M_{n,m}^{\#}(\mathcal{C}^{(2)})$ and $X$ is invertible}\}.
$$
Then both $(GL^{\#}_{n,m}(V),\, \cdot)$ and $(GL^{\#}_{n,m}(\mathcal{C}^{(2)}), \, \#)$ are groups, where the 
group product $\cdot$ for
$GL^{\#}_{n,m}(V)$ is the composition of maps and the group product $\#$ for $GL^{\#}_{n,m}(\mathcal{C}^{(2)})$ is the 
Hu-Liu matrix product.  A basic fact about these groups is  that the two groups $(GL^{\#}_{n,m}(V),\, \cdot)$ and 
$(GL^{\#}_{n,m}(\mathcal{C}^{(2)}), \, \#)$ are isomorphic.

\medskip
In this section,  $G$ always denotes a finite group.  We introduce the concept of zero divisor representation of finite groups 
in the following

\medskip
\begin{definition}\label{def4.1} Let $V$ be a  $(n, m)$-dimensional $\mathcal{C}^{(2)}$-module, and let $G$ be a finite 
group. A {\bf  zero divisor complex representation } of $G$ in $V$  is a homomorphism $\rho$ from the group $G$ into 
the group $GL^{\#}_{n,m}(V)$ or the group $GL^{\#}_{n,m}(\mathcal{C}^{(2)})$, in which case, we also say that $V$ is a
{\bf $G$-invariant  $\mathcal{C}^{(2)}$-module determined by $\rho$}.
\end{definition}

\medskip
Let $\rho$ be a zero divisor complex representation  of $G$ in a finite dimensional $\mathcal{C}^{(2)}$-module $V$.
If $U$ is a  $\mathcal{C}^{(2)}$-submodule $U$ of $V$ and $ \rho (x)(U)\subseteq U$ for all $x\in G$, then $U$ is called a 
{\bf $G$-invariant  $\mathcal{C}^{(2)}$-submodule} of $V$. Clearly,  each $\mathcal{C}^{(2)}$-submodule in the set
$\left\{\{0\}, \, Im\,1^{\#}_V, \,Ker\,1^{\#}_V, \, V \right\}$ is a $G$-invariant  $\mathcal{C}^{(2)}$-submodule of $V$.

\medskip
To introduce the building blocks in the zero divisor complex 
representation  of finite groups, we need the following version of Maschke’s theorem.

\medskip
\begin{proposition}\label{pr4.1} Let $\rho$ be a zero divisor complex representation  of a finite group $G$ in a finite dimensional 
$\mathcal{C}^{(2)}$-module $V$. If U is a $G$-invariant  $\mathcal{C}^{(2)}$-submodule of $V$
and there
exists a   $\mathcal{C}^{(2)}$-submodule W of $V$ such that $V = U \oplus W$ as 
$\mathcal{C}^{(2)}$-modules, then U has a $G$-invariant $\mathcal{C}^{(2)}$-submodule complement, i.e., there
exists a  $G$-invariant $\mathcal{C}^{(2)}$-submodule $\overline{U}$ such that 
$V = U \oplus \overline{U}$ as $G$-invariant  $\mathcal{C}^{(2)}$-modules.
\end{proposition}

\medskip
\noindent
{\bf Proof} Use the same proof as the proof of Maschke’s theorem.

\hfill\raisebox{1mm}{\framebox[2mm]{}}

\bigskip
Unlike the ordinary representation of finite groups, there are two types of building blocks in the zero divisor complex 
representation  of finite groups.

\medskip
\begin{definition}\label{def4.2} Let $\rho$ be a zero divisor complex representation  of a finite group $G$ in a finite dimensional 
$\mathcal{C}^{(2)}$-module $V$, and let $W$ be a $G$-invariant  $\mathcal{C}^{(2)}$-submodule of $V$.
\begin{description}
\item[(i)] $W$ is  {\bf $2$-irreducible} or {\bf trvial  irreducible} if $W=Ker\,1^{\#}_W\ne \{0\}$ and the only 
$G$-invariant  $\mathcal{C}^{(2)}$-submodule of $W$ are $\{0\}$ and $W$.
\item[(ii)] $W$ is  {\bf $3$-irreducible} or {\bf non-trvial  irreducible} if $ Ker\,1^{\#}_W= Im\,1^{\#}_W\ne \{0\}$ and the only 
$G$-invariant  $\mathcal{C}^{(2)}$-submodule of $W$ are $\{0\}$, $Ker\,1^{\#}_W= Im\,1^{\#}_W$ and $W$.
\end{description}
\end{definition}

\medskip
The following proposition gives the fundamental fact in  zero divisor complex representation  of a finite groups.

\medskip
\begin{proposition}\label{pr4.2} If $\rho$ ia a zero divisor complex representation  of a finite group $G$ in a finite dimensional 
$\mathcal{C}^{(2)}$-module $V$, then $V$ is a direct sum of $k$- irreducible $\mathcal{C}^{(2)}$-submodule of $V$, where $k=2$ or $3$.
\end{proposition}

\medskip
\noindent
{\bf Proof} We prove Proposition \ref{pr4.1} by cases.

\bigskip
\underline{\underline{Case 1:  $Ker\,1^{\#}_V=Im\,1^{\#}_V\ne 0$}}, in which case, 
$V$ as the direct sum of $3$- irreducible $\mathcal{C}^{(2)}$-submodule of $V$
in Case 1. Hence, Proposition \ref{pr4.2} holds.

\bigskip
\underline{\underline{Case 2:  $Ker\,1^{\#}_V\ne Im\,1^{\#}_V$}}. If $Im\,1^{\#}_V=\{0\}$, then $V$ is 
a trivial $\mathcal{C}^{(2)}$-module. Hence, $\rho$ is an ordinary complex representation  of a finite 
group $G$ in a finite dimensional $\mathcal{C}$-space $V$, which implies that $V$ is a direct sum of
$2$- irreducible $\mathcal{C}$-submodule of $V$. 

\medskip
If $Im\,1^{\#}_V\ne \{0\}$, then $\rho|Ker\,1^{\#}_V$ is an ordinary complex representation  of a finite 
group $G$ in a finite dimensional $\mathcal{C}$-space $Ker\,1^{\#}_V$, and $Im\,1^{\#}_V$ is a 
$G$-invariant  $\mathcal{C}$-subspace of $Ker\,1^{\#}_V$. By Maschke's theorem, there exists a $G$-invariant  
$\mathcal{C}$-subspace $H$ of $Ker\,1^{\#}_V$ such that
$Ker\,1^{\#}_V=Im\,1^{\#}_V\oplus H$ as $G$-invariant  $\mathcal{C}$-spaces.
Let $\{1^{\#}u_i^{(2)}\}_{i=1}^p$ be a $\mathcal{C}$-basis of $Im\,1^{\#}_V$ and let 
$W=\displaystyle\bigoplus_{i=1}^p\,\mathcal{C}^{(2)}u_i^{(2)}$. Then
\begin{equation}\label{eq342}
V=H\oplus W
\quad\mbox{(as $\mathcal{C}^{(2)}$-modules).}
\end{equation}
It follows from Proposition \ref{pr4.1} and (\ref{eq342})  that there
exists a  $G$-invariant $\mathcal{C}^{(2)}$-submodule $\overline{H}$ such that 
\begin{equation}\label{eq343}
V=H\oplus  \overline{H}
\quad\mbox{(as $G$-invariant  $\mathcal{C}^{(2)}$-modules).}
\end{equation}
Since $ \overline{H}\stackrel{(\ref{eq343})}{\simeq}\displaystyle\frac{V}{H}
\stackrel{(\ref{eq342})}{\simeq}W$ as $\mathcal{C}^{(2)}$-modules and 
 $Ker\,1^{\#}_W=Im\,1^{\#}_W\ne 0$, we have 
$Ker\,1^{\#}_{\overline{H}}=Im\,1^{\#}_{\overline{H}}\ne 0$. By Case 1. we get
\begin{equation}\label{eq344}
\mbox{$\overline{H}$ is a direct sum of $3$- irreducible $\mathcal{C}^{(2)}$-submodule of $V$.}
\end{equation}
Since the ordinary representation $\rho|H$ of $G$   is complete reducible, we have
\begin{equation}\label{eq345}
\mbox{$H$ is a direct sum of $2$- irreducible $\mathcal{C}^{(2)}$-submodule of $V$.}
\end{equation}
By (\ref{eq343}), (\ref{eq344}) and (\ref{eq345}),  Proposition \ref{pr4.2} holds in Case 2.

\hfill\raisebox{1mm}{\framebox[2mm]{}}

\bigskip

\section{Appendix:\,\, A class of  representations of non-semisimple Lie algebras}

In this appendix, we introduce  the concept of  respecting nilpotent step representations of 
Lie algebras and give the counterpart of Ado's Theorem in the context of  respecting nilpotent 
step representations of Lie algebras. Throughout  this appendix,  all vector spaces are the vector 
spaces  over algebraically closed fields of  characteristic $0$. 

\medskip
Let $V=\displaystyle\bigoplus_{i=0}^n V_i$ be a finite dimensional graded vector space. As usual, we use 
$End(V)$ to denote the set of all linear transformations from $V$ to $V$. Let
$$
End_{\bigtriangleup}^i(V):=\left\{f \left | \mbox{$f\in End(V)$ and $f(V_k)\subseteq  
\displaystyle\bigoplus_{t=k}^{n-i} V_{t+i}$ for $0\le k\le n$}\right.\right\},
$$
where $0\le i\le n$ and $V_t:=0$ for $t\ge n+1$. 

\medskip
We say that a Lie algebra $L$ is a Lie algebra with the {\bf nilpotent step} $n+1$ if  the
lower central series of the nil radical $nil\,L$ of the Lie algebra $L$ is given  as follows:
$$
nil\,L\supset (nil\,L)^2\supset\cdots\supset (nil\,L)^n\supset (nil\,L)^{n+1}=0.
$$ 

\medskip
\begin{definition}\label{def1.1} Let $L$ be a finite-dimensional Lie algebra with the nilpotent step $n+1$. We say 
that $\varphi$ is a {\bf  respecting nilpotent step representation of the Lie algebra $L$} if there exists a finite-dimensional graded vector space  $V=\displaystyle\bigoplus_{i=0}^n V_i$ such that $\varphi: L\to  gl(V)$ is a Lie algebra homomorphism satisfying the following two conditions:
\begin{description} 
\item[(i)] $\varphi(L_{semi})\subseteq \{f\,|\, \mbox{$f\in End(V)$ and $f(V_k)\subseteq V_k$ for $k=0, 1, \dots, n$}\}$, 
where \linebreak $L_{semi}$ is a Levi factor of the Lie algebra $L$.
\item[(ii)] $\varphi\big((nil L)^i\big)\subseteq End_{\bigtriangleup}^i(V)$ for $1\le i\le n$.
\end{description}
\end{definition}

\medskip
The  counterpart of Ado's Theorem in the context of  respecting nilpotent step representations of Lie algebras is the following

\medskip
\begin{proposition}\label{pr5.1} Every finite-dimensional Lie algebra over an algebraically closed field of  characteristic $0$ has a faithful finite-dimensional respecting nilpotent step representation. 
\end{proposition}

\bigskip
{\bf Acknowledgement} Most of the results in this paper were obtained while Professors Dale Rolfsen, George Bluman, and Brian Marcus served as Head of the Department of Mathematics; Professor Ed Perkins served as Associate Head for Faculty Affairs; and Professor Charles Lamb served as Undergraduate Chair. I am deeply grateful to them for their assistance and support. In particular, Professor Lamb’s outstanding leadership of the undergraduate teaching program created an environment that enabled me to devote substantial time to mathematical research. Finally, I am especially grateful to Professor R. V. Moody for his valuable guidance,  support, and encouragement.

\end{document}